\documentclass{article}
\usepackage[utf8]{inputenc}
\usepackage{amssymb,amsmath, amsthm, mathabx}
\usepackage{amsmath,amsfonts,amssymb,amsthm,epsfig,epstopdf,titling,url,array}
\usepackage{tikz}
\usetikzlibrary{shapes,snakes}
\usepackage{color, enumerate}
\usepackage{comment, authblk}
\usepackage{hyperref}
\usepackage{pgfplots}
\usepackage{bm}
\usepackage{stmaryrd}
\usepackage{graphicx}
\usepackage{caption}
\usepackage{subcaption}
\usepackage{a4wide}
\usepackage{titlesec}
\usepackage[a4paper, total={6in, 8in}]{geometry}

\titleformat*{\subsection}{\large\bfseries}
\numberwithin{equation}{section}

\pgfplotsset{compat=newest}
\pgfplotsset{plot coordinates/math parser=false}
\newlength\figureheight
\newlength\figurewidth

\newtheorem{thm}{Theorem}[section]

\newtheorem{prop}[thm]{Proposition}
\newtheorem{lem}[thm]{Lemma}

\newtheorem{defn}[thm]{Definition}

\newtheorem{remark}[thm]{Remark}

\renewcommand{\Im}{{\rm{Im}}}
\renewcommand{\Re}{{\rm{Re}}}

\allowdisplaybreaks

\title{The smallest singular value of deformed random rectangular matrices}

\author[1]{Fan Yang  \thanks{E-mail: fyang75@math.wisc.edu. Partially supported by NSF Career Grant DMS-1552192.}}

\affil[1]{Department of Mathematics, University of Wisconsin-Madison}

\begin{document}
\maketitle

\begin{abstract}
We prove an estimate on the smallest singular value of a multiplicatively and additively deformed random rectangular matrix. Suppose $n\le N \le M \le dN$ for some constant $d\ge 1$. Let $X$ be an $M\times n$ random matrix with independent and identically distributed entries, which have zero mean, unit variance and arbitrarily high moments. Let $T$ be an $N\times M$ deterministic matrix with comparable singular values $c\le s_{N}(T) \le s_{1}(T) \le c^{-1}$ for some constant $c>0$. Let $A$ be an $N\times n$ deterministic matrix with $\|A\|=O(\sqrt{N})$. Then we prove that for any $\epsilon>0$, the smallest singular value of $TX-A$ is larger than $N^{-\epsilon}(\sqrt{N}-\sqrt{n-1})$ with high probability. If we assume further the entries of $X$ have subgaussian decay, then the smallest singular value of $TX-A$ is at least of the order $\sqrt{N}-\sqrt{n-1}$ with high probability, which is an essentially optimal estimate.
\end{abstract}

\section{Introduction}

\subsection{Smallest singular values of random matrices}


Consider an $N\times n$ real or complex matrix $A$. The singular values $s_i(A)$ of $A$ are the eigenvalues of $(A^* A)^{1/2}$ arranged in the non-increasing order:
$$s_1(A) \ge s_2(A) \ge \ldots \ge s_n(A).$$
Of particular importance are the largest singular value $s_1(A)$, which gives the spectral norm $\|A\|$, and the smallest singular value $s_n(A)$, which measures the invertibility of $A^* A$ in the $N\ge n$ case. 

A natural random matrix model is given by a rectangular matrix $X$ whose entries are independent random variables with mean zero, unit variance and certain moment assumptions. In this paper, we focus on random variables with {\it{arbitrarily high moments}} (see (\ref{assm2})), which include all the {\it{subgaussian}} and {\it{subexponential}} random variables. 
The asymptotic behavior of the extreme singular values of $X$ has been well-studied. Suppose $X$ has dimensions $N\times n$. Let $\lambda_1\ge \lambda_2 \ge \ldots \ge \lambda_n$ be the eigenvalues of $N^{-1}X^* X$ and define the empirical spectral distribution as $\mu_N := n^{-1}\sum_{i=1}^n \delta_{\lambda_i}$. If $n/N\to \lambda \in (0,1)$ as $N\to \infty$, then $\mu_N$ converges weakly to the famous March{\v e}nko-Pastur (MP) law \cite{MP}. Moreover, the MP distribution has a density with positive support on $[(1-\sqrt{\lambda})^2, (1+\sqrt{\lambda})^2]$, which suggests that asymptotically,
\begin{equation}\label{asymp}
s_1(X) \to \sqrt{N}(1+\sqrt{\lambda})= \sqrt{N} + \sqrt{n}, \ \text{ and }\ s_n(X) \to \sqrt{N}(1-\sqrt{\lambda})= \sqrt{N} - \sqrt{n}.
\end{equation} 
The almost sure convergence of the largest singular value was proved in \cite{Geman_large} for random matrices whose entries have arbitrarily high moments. The almost sure convergence of the smallest singular value was proved in \cite{Silverstein_small} for Gaussian random matrices (i.e. the Wishart matrix). These results were later generalized to random matrices with $i.i.d.$ entries with finite fourth moment in \cite{BaiYin_large} and \cite{BaiYin_law}.

A considerably harder problem is to establish non-asymptotic versions of (\ref{asymp}), which would hold for any fixed dimensions $N$ and $n$. Most often needed are upper bounds for the largest singular value $s_1(X)$ and lower bounds for the smallest singular value $s_n(X)$. With a standard $\epsilon$-net argument, it is not hard to prove that $\|X\|$ is at most of the optimal order $\sqrt{N}$ for all dimensions, see e.g. \cite{Handbook_DS,Random_polytopes,RudVersh_rect}. On the other hand, the smallest singular value is much harder to bound below. There has been much progress in this direction during the last decade. 

{\bf Tall matrices.} It was proved in \cite{Rud_polytope} that for arbitrary aspect ratios $\lambda < 1 - c/\log N$ and for random matrices with independent subgaussian entries, one has 
\begin{equation}\label{optimal_tall}
\mathbb P\left(s_n(X) \le c_\lambda \sqrt{N}\right) \le e^{-cN},
\end{equation}
where $c_\lambda>0$ depends only on $\lambda$ and the maximal subgaussian moment of the entries. 

{\bf Square matrices.} For square random matrices with $N=n$, a lower bound for the smallest singular value was first obtained in \cite{Rud_Annal}, where it was proved that for subgaussian random matrix $X$, $s_N(X)\ge \epsilon N^{-3/2}$ with high probability. This result was later improved in \cite{RudVersh_square} to
\begin{equation}\label{optimal_square}
\mathbb P\left(s_N(X) \le \epsilon N^{-1/2}\right) \le C\epsilon + e^{-cN} ,
\end{equation} 
an essentially optimal estimate for subgaussian matrices. Subsequently, different lower bounds for $s_N(X)$ were proved under weakened moments assumptions \cite{TaoVu_circular,PanZhou_circular,gotze2010}. 

{\bf Almost square matrices.} The gap $1 - c/\log N \le \lambda < 1$ was filled in \cite{RudVersh_rect}. It was shown that for subgaussian random rectangular matrices, 
\begin{equation}\label{optimal_rect}
\mathbb P\left(s_n(X) \le \epsilon (\sqrt{N} - \sqrt{n-1})\right) \le \left( C\epsilon\right)^{N-n+1} + e^{-cN} ,
\end{equation}
for all fixed dimensions $N\ge n$. This bound is essentially optimal for subgaussian matrices with all aspect ratios. It is easy to see that (\ref{optimal_tall}) and (\ref{optimal_square}) are the special cases of the estimate (\ref{optimal_rect}). 


In this paper, we are interested in the extreme singular values of a multiplicatively and additively deformed random rectangular matrix. Given an $M\times n$ random matrix $X$ with independent entries, we consider the matrix $TX-A$, where $T$ and $A$ are $N\times M$ and $N\times n$ deterministic matrices, respectively. It is easy to bound above the largest singular value using $\|TX-A\| \le \|T\|\|X\| + \|A\|$. On the other hand, we expect that if $n\le N\le M$ and the singular values of $T$ satisfy $c\le s_N(T) \le s_1(T) \le c^{-1}$, then a similar estimate as in (\ref{optimal_rect}) would still hold for $TX-A$. In fact, if $M=N$ and $X$ is subgaussian, one can prove that the estimate (\ref{optimal_rect}) holds for the matrix $X- T^{-1}A$ with a direct generalization of the method in \cite{RudVersh_rect}. Together with $s_n(TX-A) \ge s_N(T) s_n(X-T^{-1}A)$, this already gives the desired lower bound for $s_n(TX-A)$. In this paper, we will consider more general case where $N\le M$ and $X$ is not necessarily subgaussian, see Theorem \ref{main_small}.


One of our motivations is the potential application in statistical science. Consider sample covariance matrices of the form $Q= n^{-1}B B^*$, where $B$ is an $N\times n$ matrix. The columns of $B$ represent $n$ independent observations of some random $N$-dimensional vector $\mathbf b$. For the sample vector $\mathbf b$, we take a linear model $\mathbf b =T\mathbf x$, where $T$ is a deterministic $N\times M$ matrix and $\mathbf x$ is a random $M$-dimensional vector with independent entries. Then we can write $\mathbf b = T\hat{\mathbf x} + \mathbf a$, where $\hat{\mathbf x}$ is a centered random vector and $\mathbf a =T\mathbb E\mathbf x$. In addition, without loss of generality, we may assume that the entries of $\hat{\mathbf x}$ have unit variance by absorbing the variance of $\hat x_i$ into $T$. Hence we can write $B$ into the form $B= TX - A$, and our result would provide a good a priori estimate on the smallest singular values of $B$. 

Another application of our result is the {\it{circular law}} for square random matrices. Let $X$ be an $N\times N$ random matrix with $i.i.d.$ entries with zero mean and unit variance. It is well known that the spectral measure of eigenvalues of $N^{-1/2}X$ converges to the circular law, i.e. the uniform distribution on the unit disk \cite{Ginibre,Bai1997}. An important input of the proof is the lower bound for the smallest singular value of $N^{-1/2}X - z$ for any $z\in \mathbb C$ \cite{gotze2010,PanZhou_circular,TaoVu_circular,tao2010}. In \cite{XYY}, we proved a generalized {\it{local circular law}} for square random matrices of the form $N^{-1/2}TX$. In order to obtain a lower bound for the smallest singular value of $N^{-1/2}TX - z$, we assumed that the entries of $X$ have continuous distributions. This assumption rules out some important models such as the Bernoulli random matrices. Now with the result of this paper, we can relax this assumption greatly to include all random variables with sufficiently high moments.

\subsection{Main result and the reduction to subgaussian matrices}

Let $\xi_1, \ldots , \xi_n$ be independent random variables such that for $1\le i \le n$,
\begin{equation}\label{assm1}
 \mathbb E \xi_{i}=0,\ \ \mathbb E|\xi_{i}|^2=1, 
\end{equation}
and for any $p\in\mathbb N$, there is an $N$-independent constant $\sigma_p$ such that
\begin{equation}\label{assm2}
 \mathbb E|\xi_{i}|^p\le \sigma_p.
\end{equation}
We assume that $X$ is an $M\times n$ random matrix, whose rows are independent copies of the random vector $(\xi_1,\ldots, \xi_n)$. In this paper, we consider the deformed random rectangular matrix $TX - B$, where $T$ and $B$ are $N\times M$ and $N\times n$ deterministic matrices, respectively. We assume that
\begin{equation}\label{assm0}
n \le N \le M \le \Lambda N, \ \ \|B\| \le K_0\sqrt{N}
\end{equation}
for some constants $K_0,\Lambda\ge 1$. Moreover, we assume the eigenvalues of $TT^*$ satisfy that
\begin{equation}\label{assm3}
K_0^{-1} \le \sigma_N \le \ldots \le \sigma_{2} \le \sigma_{1} \le K_0.
\end{equation}
For definiteness, in this paper we focus on the case with {\it{real}} matrices. However, our results and proof also hold, after minor changes, in the {\it{complex}} case if we assume in addition that $X_{ij}$ have independent real and imaginary parts, such that
$$ \mathbb E \left(\Re\, X_{ij}\right)=0,\ \ \mathbb E\left(\Re\, X_{ij}\right)^2=\frac{1}{2}, $$
and similarly for $\Im\, X_{ij}$. The main result of this paper is the following theorem. 

\begin{thm}\label{main_small}
Suppose the assumptions (\ref{assm1}), (\ref{assm2}), (\ref{assm0}) and (\ref{assm3}) hold. Fix any constants $\tau>0$ and $\Gamma>0$. Then for every $\epsilon \ge 0$, we have
\begin{equation}\label{main_eq_small1}
\mathbb P\left(s_n(TX-B) \le \epsilon N^{-\tau} \left(\sqrt N-\sqrt{n-1}\right)\right) \le (C\epsilon)^{N-n+1} + N^{-\Gamma}
\end{equation}
for large enough $N \ge N_0$, where the constant $C>0$ depends only on $\sigma_p$, $\Lambda$ and $K_0$, and $N_0$ depends only on $\sigma_p$, $\Lambda$, $\Gamma$ and $\tau$. 
\end{thm}

To prove this theorem, we first truncate the entries of $X$ at level $N^\omega$ for some small $\omega>0$. Combining condition (\ref{assm2}) with Markov's inequality, we get that for any (small) $\omega>0$ and (large) $\Gamma>0$, there exists $N(\omega, \Gamma)$ such that 
$$\mathbb P\left( |\xi_i| > N^{\omega}/2\right) \le N^{-\Gamma-2}$$
for all $N\ge N(\omega,\Gamma)$. Hence with a loss of probability $O(N^{-\Gamma})$, it suffices to control the smallest singular values of the random matrix $T\tilde X - B$, where
$$\tilde X :=1_{\Omega} X, \ \ \Omega :=\left\{|X_{ij}|\le N^{\omega}/2 \text{ for all } 1\le i \le M, 1\le j \le n\right\}.$$
By (\ref{assm2}) and integration by parts, we can check that for $1\le i \le n$,
\begin{equation}\label{cutoff_bound}
\mathbb E\left( \xi_i {\mathbf 1}_{\{|\xi_i| \le N^{\omega}/2\}} \right)= O\left(N^{-\Gamma-2+\omega}\right), \ \ \text{Var}\left(\xi_i {\mathbf 1}_{\{|\xi_i| \le N^{\omega}/2\}}\right) = 1 + O\left(N^{-\Gamma-2+2\omega}\right).
\end{equation}
We define $D_1$ to be an $n\times n$ diagonal matrix with $(D_1)_{ii} = \text{Var}\left(\xi_i {\mathbf 1}_{\{|\xi_i| \le N^{\omega}/2\}}\right)^{1/2}$. 

Let $T=U\tilde DV$ be a singular value decomposition of $T$, where $U$ is an $N\times N$ unitary matrix, $V$ is an $M\times M$ unitary matrix and $\tilde D=(D,0)$ is an $N\times M$ rectangular diagonal matrix such that $D=\text{diag}(d_1,d_2,\ldots,d_N)$ with $d_i^2=\sigma_i$. 
We denote $V = \left( {\begin{array}{*{20}c}
   { V_1 }  \\
   {V_2}  \\
\end{array}} \right),$ where $V_1$ has size $N\times M$ and $V_2$ has size $(M-N)\times M$. Then we have
\begin{equation*}
T\tilde X - B = UDV_1(\tilde X - \mathbb E\tilde X) - (B - T\mathbb E\tilde X) = UD\left[V_1 (\tilde X - \mathbb E\tilde X)D_1^{-1}  - \left(D^{-1}U^{-1}B - V_1 \mathbb E\tilde X\right)D_1^{-1} \right]D_1. 
\end{equation*}
Due to (\ref{assm3}) and (\ref{cutoff_bound}), we only need to bound $s_n(V_1 Y - A)$, where 
$$Y:=(\tilde X - \mathbb E\tilde X)D_1^{-1}, \ \ \text{and} \ \ A:=(D^{-1}U^{-1}B - V_1 \mathbb E\tilde X)D_1^{-1}.$$ 
Using (\ref{assm0}), (\ref{assm3}), (\ref{cutoff_bound}) and the definition of $\Omega$, it is easy to check that $A$ is a deterministic matrix with
\begin{equation}\label{prop_Y2}
\|A\|\le C\left(\|B\| + \|\mathbb E\tilde X\|\right) \le C\left(\sqrt{N}+ N^{-\Gamma-1+\omega}\right) \le C\sqrt{N},
\end{equation}
and $Y$ is a random matrix with independent entries satisfying
\begin{equation}\label{prop_Y1}
\mathbb E(Y_{ij}) = 0, \ \ \text{Var}(Y_{ij}) = 1, \ \ |Y_{ij}| \le N^\omega.
\end{equation}
Recall that a random variable $\xi$ is called subgaussian if there exists $K>0$ such that
\begin{equation}\label{assm_sub}
\mathbb P\left( |\xi| > t \right)\le 2 \exp(-t^2/K^2) \ \ \text{for all }t>0.
\end{equation}
The infimum of such $K$ is called the subgaussian moment of $\xi$ or the $\psi_2$-norm $\|\xi\|_{\psi_2}$. By (\ref{prop_Y1}), it is obvious that $Y_{ij}$ are subgaussian random variables with $\|Y_{ij}\|_{\psi_2}\le N^{\omega}$.
Moreover, by Theorem 2.10 of \cite{isotropic}, there exists a constant $C>0$ such that
$$\mathbb P \left( \|X\| \le C\sqrt{N} \right) \ge 1- N^{-\Gamma}$$
for large enough $N$. Then using $\|\tilde X\|\le \|X\|$, we get that
\begin{equation}\label{prop_Y3}
\mathbb P(\| Y \| \le C\sqrt{N}) \ge 1 - N^{-\Gamma}.
\end{equation}
From the above discussion, we see that Theorem \ref{main_small} follows from the following theorem.

\begin{thm}\label{main_sub}
Let $\xi_1, \ldots , \xi_n$ be independent centered random variables with unit variance, finite fourth moments and subgaussian moments bounded by $K$ for some $K\equiv K(N) \le N^{\omega}$. Let $Y$ be an $M\times n$ random matrix, whose rows are independent copies of the random vector $(\xi_1,\ldots,\xi_n)$. Let $P$ be an $N\times M$ deterministic matrix with $PP^T = 1$, and let $A$ be an $N\times n$ deterministic matrix.  Suppose that $\|Y\| + \|A\| \le C_1 \sqrt{N}$ for some constant $C_1>0$. Then for every $0< \omega < \omega_0$ and every $\epsilon \ge 0$, we have
\begin{equation}\label{main_eq_small}
\mathbb P\left(s_n(PY-A) \le \epsilon \left(\sqrt N-\sqrt{n-1}\right)\right) \le \left(CK^{L} \epsilon \right)^{N-n+1} + e^{-cN/K^{4}},
\end{equation}
where the constants $\omega_0, c, C, L>0$ depend only on $\Lambda$, $C_1$ and the maximal fourth moment.
\end{thm}


\begin{remark}
Suppose $X_{ij}$ are subgaussian random variables with $\max_{i,j}\|X_{ij}\|\le K$ for some constant $K>0$. Then we have
$$\mathbb P(\|X\|\ge t \sqrt{N}) \le e^{-c_0 t^2 N} \ \ \text{ for } t\ge C_0,$$
where $c_0,C_0>0$ depend only on $K$ (see \cite[Proposition 2.4]{RudVersh_rect}). Combining with Theorem \ref{main_sub}, we obtain the optimal estimate for the smallest singular value of $TX-B$:
\begin{equation} 
\mathbb P\left(s_n(TX-B) \le \epsilon \left(\sqrt N-\sqrt{n-1}\right)\right) \le (C\epsilon)^{N-n+1} + e^{-cN}.
\end{equation}
\end{remark}



The bulk of this paper is devoted to the proof of Theorem \ref{main_sub}. In the preliminary Section \ref{sec_preliminary}, we introduce some notations and tools that will be used in the proof. In Section \ref{sec_decompose}, we first reduce the problem into bounding below $\|(PY-A)x\|_2$ for {\it{compressible unit vectors}} $x \in S^{n-1}$, whose $l^2$-norm is concentrated in a small number of coordinates, and for {\it{incompressible unit vectors}} comprising the rest of the sphere $S^{n-1}$. Then we prove a lower bound for compressible unit vectors using a small ball probability result (Lemma \ref{lem_smallball}) and a standard $\epsilon$-net argument. The incompressible unit vectors are dealt with in Sections \ref{sec_tall} and \ref{sec_square}. In Section \ref{sec_tall}, we consider the case $1\le n \le \lambda_0 N$ for some constant $\lambda_0 \in (0,1)$, i.e. when $PY-A$ is a tall matrix. The proof can be finished with another small ball probability result (Lemma \ref{small_CLT}) and the $\epsilon$-net argument. The almost square case with $\lambda_0N < n \le N$ is considered in Section \ref{sec_square}. We first reduce the problem into bounding the distance between a random vector and a random subspace, and then complete the proof with a random distance lemma---Lemma \ref{dist_lemm}, whose proof will be given in Section \ref{section_distance}.

\vspace{5pt}

\noindent{\bf Acknowledgements.} I would like to thank Roman Vershynin for valuable suggestions and for pointing out several useful references. I am also grateful to my advisor Jun Yin for his financial support and helpful comments on this paper. Finally, many thanks to my friend Haokai Xi for useful discussions.

\section{Basic notations and tools}\label{sec_preliminary}

In this paper, we use $C$ to denote a generic large positive constant, which may depend on fixed parameters and whose value may change from one line to the next. Similarly, we use $c$, $\epsilon$ or $\omega$ to denote a generic small positive constant. If a constant depends on a quantity $a$, we use $C(a)$ or $C_a$ to indicate this dependence. 

The canonical inner product on $\mathbb R^n$ is denoted $\langle \cdot,\cdot\rangle$, and the Euclidean norm is denoted $\|\cdot\|_2$. The distance from a point $x$ to a set $D$ in $\mathbb R^n$ is denoted $\text{dist}(x,D)$. The unit sphere centered at the origin in $\mathbb R^n$ is denoted $S^{n-1}$. The orthogonal projection in $\mathbb R^n$ onto a subspace $E$ is denoted $P_E$. For a subset of coordinates $J \subseteq \{1,\ldots, n\}$, we often write $P_J$ for $P_{\mathbb R^J}$. The unit sphere of $E$ is denoted $S(E):=S^{n-1} \cap E$. 

For any matrix $A$, we use $A^*$ to denote its conjugate transpose, $A^T$ the transpose, $\|A\| := \|A\|_{l^2 \to l^2}$ the operator norm and $\|A\|_{HS}$ the Hilbert-Schmidt norm. We usually write an identity matrix as $1$ without causing any confusions.


The following tensorization lemma is Lemma 2.2 of \cite{RudVersh_square}

\begin{lem}[Tensorization]\label{lem_tensor}
Let $\zeta_1,\ldots,\zeta_n$ be independent non-negative random variables, and let $B,\epsilon_0\ge 0$. 

\begin{itemize}
\item[(1)]
Assume that for each $k$,
$$\mathbb P(\zeta_k <\epsilon) \le B\epsilon \ \ \text{for all }\epsilon \ge \epsilon_0.$$
Then
$$\mathbb P\left( \sum_{k=1}^n \zeta_k^2 < \epsilon^2 n\right)\le \left(CB\epsilon\right)^n \ \ \text{for all }\epsilon \ge \epsilon_0,$$
where $C$ is an absolute constant.

\item[(2)] Assume that there exist $\lambda>0$ and $\mu\in (0,1)$ such that for each $k$,
$$\mathbb P(\zeta_k <\lambda) \le \mu.$$
Then there exists $\lambda_1>0$ and $\mu_1\in (0,1)$ that depend on $\lambda$ and $\mu$ only and such that
$$\mathbb P\left( \sum_{k=1}^n \zeta_k^2 < \lambda_1 n\right)\le \mu_1^n.$$
\end{itemize}
\end{lem}

Consider a subset $\Omega\subset \mathbb R^n$, and let $\epsilon>0$. An $\epsilon$-net of $\Omega$ is a subset $\mathcal N \subseteq \Omega$ such that for every $x\in \Omega$ one has $\text{dist}(x,\mathcal N)\le \epsilon$. The following lemma is proved as Propositions 2.1 and 2.2 in \cite{RudVersh_rect}.

\begin{lem}[Nets]\label{card_nets}
Fix any $\epsilon>0$.
\begin{itemize}
\item[(1)] There exists an $\epsilon$-net of $S^{n-1}$ of cardinality at most 
$$\min\left\{\left(1+2\epsilon^{-1}\right)^{n},2n\left(1+2\epsilon^{-1}\right)^{n-1}\right\}.$$

\item[(2)] Let $S$ be a subset of $S^{n-1}$. There exists an $\epsilon$-net of $S$ of cardinality at most 
$$\min\left\{\left(1+4\epsilon^{-1}\right)^{n},2n\left(1+4\epsilon^{-1}\right)^{n-1}\right\}.$$
\end{itemize}
\end{lem}

Next we define the small ball probability for a random vector.

\begin{defn}\label{defn_small_ball}
The L{\'e}vy concentration function of a random vector $S\in \mathbb R^m$ is defined for $\epsilon>0$ as
$$\mathcal L(S,\epsilon) = \sup_{v\in \mathbb R^m} \mathbb P(\|S-v\|_2 \le \epsilon),$$
which measures the small ball probabilities. 
\end{defn}

With Definition \ref{defn_small_ball}, it is easy to prove the following lemma. It will allow us to select a nice subset of the coefficients $a_k$ when computing the small ball probability.

\begin{lem}\label{lem_select}
Let $\xi_1,\ldots,\xi_n$ be independent random variables. For any $\sigma \subseteq \{1,\ldots, n\}$, any $a\in \mathbb R^n$ and any $\epsilon\ge 0$, we have
$$\mathcal L\left(\sum_{k=1}^n a_k \xi_k ,\epsilon\right) \le \mathcal L\left(\sum_{k\in \sigma} a_k \xi_k ,\epsilon \right).$$
\end{lem}

The following three lemmas give some useful small ball probability bounds. They correspond to \cite[Lemma 3.2]{RudVersh_rect}, \cite[Corollary 2.9]{RudVersh_square} and \cite[Corollary 2.4]{HansonW} respectively.

\begin{lem}\label{lem_Paley}
Let $\xi$ be a random variable with mean zero, unit variance, and finite fourth moment. Then for every $\epsilon \in (0,1)$, there exists a $p\in (0,1)$ which depends only on $\epsilon$ and on the fourth moment, and such that
$$\mathcal L(\xi,\epsilon) \le p.$$
\end{lem}

\begin{lem}\label{small_CLT}
Let $\xi_1, \ldots, \xi_n$ be independent centered random variables with variances at least 1 and third moments bounded by $B$. Then for every $a\in \mathbb R^n$ and every $\epsilon \ge 0$, one has
$$\mathcal L\left(\sum_{k=1}^n a_k \xi_k,\epsilon\right) \le \sqrt{\frac{2}{\pi}} \frac{\epsilon}{\|a\|_2} + \tilde CB\left( \frac{\|a\|_3}{\|a\|_2}\right)^3,$$
where $\tilde C$ is an absolute constant.
\end{lem}

\begin{lem} \label{lem_smallball}
Let $A$ be a fixed $N\times M$ matrix. Consider a random vector $\xi=(\xi_1, \ldots, \xi_M)$ where $\xi_i$ are independent random variables satisfying $\mathbb E\xi_i =0$, $\mathbb E\xi_i^2 = 1$ and $\|\xi_i\|_{\psi_2} \le K$. Then for every $y\in \mathbb R^N$, we have
$$\mathbb P\left\{ \|A\xi - y\|_2 \le \frac{1}{2} \|A\|_{HS} \right\} \le 2\exp \left( -\frac{c\|A\|_{HS}^2}{K^4\|A\|^2}\right).$$
\end{lem}

%
%

%

\section{Decomposition of the sphere}\label{sec_decompose}

Now we begin the proof of Theorem \ref{main_sub}. We will make use of a partition of the unit sphere into two sets of compressible and incompressible vectors. They are first defined in \cite{RudVersh_square}.

\begin{defn}\label{defn_compress}
Let $\delta, \rho \in (0,1]$. A vector $x\in {\mathbb R}^n$ is called sparse if $|{\rm{supp}}(x)| \le \delta n$. A vector $x\in S^{n-1}$ is called compressible if $x$ is within Euclidean distance $\rho$ from the set of all sparse vectors. A vector $x\in S^{n-1}$ is called incompressible if it is not compressible. The sets of sparse, compressible and incompressible vectors will be denoted by $\text{Sparse}_n (\delta)$, $Comp_n (\delta,\rho)$ and $Incomp_n(\delta,\rho)$. We sometimes omit the subindex $n$ when the dimension is clear.
\end{defn}
Using the decomposition $S^{n-1} = Comp \cup Incomp$, we break the invertibility problem into two subproblems, for compressible and incompressible vectors:
\begin{align}
\mathbb P \left( s_n(PY-A) \le \epsilon (\sqrt{N}- \sqrt{n-1}) \right) \le  \mathbb P\left( \inf_{x\in Comp_n(\delta,\rho)} \|(PY-A)x\|_2 \le \epsilon (\sqrt{N}- \sqrt{n-1}) \right) \\
+ \mathbb P\left( \inf_{x\in Incomp_n(\delta,\rho)} \|(PY-A)x\|_2 \le \epsilon (\sqrt{N}- \sqrt{n-1}) \right).\label{incomp_bound}
\end{align}
The bound for compressible vectors follows from the following lemma, which is a variant of Lemma 3.3 from \cite{RudVersh_square}.

\begin{lem}\label{comp_bound} 
Suppose the assumptions in Theorem \ref{main_sub} hold. Then there exist $\rho, c_0, c_1 > 0$ that depend only on $C_1$, and such that for $\delta \le \min\left\{c_1 N/\left(n K^{4}\log K\right),1\right\}$, we have 
$$ \mathbb P\left( \inf_{x\in Comp_n(\delta,\rho)} \|(PY - A)x\|_2 \le c_0 \sqrt{N}\right) \le e^{-c_0 N/K^4}.$$
\end{lem}
\begin{proof}
We first prove a similar estimate for sparse vectors. For any $x\in S^{n-1}$, we define the random vector $\zeta:= Yx \in \mathbb R^N$. It is easy to verify that $\mathbb E\zeta_i =0$, $\mathbb E\zeta_i^2 = 1$ and $\|\zeta_i\|_{\psi_2} \le CK$. Then with $\|P\|=1$ and $\|P\|_{HS}^2 = N$, we conclude from Lemma \ref{lem_smallball} that
\begin{equation}\label{incomp_key}
\mathbb P\left\{ \|(PY - A)x\|_2 \le \frac{1}{2}\sqrt{N}\right\} \le 2\exp \left( -\frac{cN}{K^4}\right).
\end{equation}
Let $S_1:=\{x\in S^{n-1}: x_k = 0 , k > \left\lceil\delta n\right\rceil\}.$ By Lemma \ref{card_nets}, there exists an $\epsilon$-net $\mathcal N$ of $S_1$ with $|\mathcal N| \le (5/\epsilon)^{\left\lceil\delta n\right\rceil}$. Then using (\ref{incomp_key}) and taking the union bound, we get
\begin{equation} \label{union1}
\mathbb P\left( \inf_{ x\in \mathcal N} \|(PY-A)x\|_2 \le \frac{1}{2}\sqrt{N}\right) \le 2e^{ - {cN}/{K^4}}\left({5}{\epsilon}^{-1}\right)^{\left\lceil\delta n\right\rceil}. 
\end{equation}
Let $V$ be the event that $\|(PY-A)y\|\le \sqrt{N}/4$ for some $y\in S_1$. By the assumptions of Theorem \ref{main_sub}, we have
$$\|PY-A\|\le \|Y\| + \|A\| \le C_1 \sqrt{N}.$$ 
Assume that $V$ occurs and choose a point $x\in \mathcal N$ such that $\|y-x\|\le \epsilon$. Then 
$$\|(PY - A)x\|_2 \le \|(PY-A)y\|_2 + \|PY-A\|\|x-y\|_2 \le \frac{1}{4}\sqrt{N} + C_1\epsilon \sqrt{N} \le \frac{1}{2}\sqrt{N},$$
if we choose $\epsilon \le 1/(4C_1)$. Fix one such $\epsilon$, using (\ref{union1}) we obtain that 
$$ \mathbb P\left( \inf_{x\in S_1} \|(PY-A)x\|_2 \le \frac{1}{4} \sqrt{N}\right) = \mathbb P(V) \le 2e^{-{cN}/{K^4}} \left(5\epsilon^{-1}\right)^{\left\lceil\delta n\right\rceil} \le e^{-c_2 N/K^4},$$ 
if we choose $c_1$ (and hence $\delta$) to be sufficiently small. We use this result and take the union bound over all $\lceil \delta n\rceil$-element subsets $\sigma$ of $\{1,\ldots, n\}$:
\begin{align}
& \mathbb P \left( \inf_{x\in Sparse(\delta)\cap S^{n-1}} \|(PY-A)x\|_2 \le \frac{1}{4} \sqrt{N}\right) \nonumber\\
= & \mathbb P \left( \exists \sigma, |\sigma|=\lceil \delta n\rceil: \inf_{x\in \mathbb R^{\sigma}\cap S^{n-1}} \|(PY-A)x\|_2 \le \frac{1}{4} \sqrt{N}\right) \nonumber\\
\le & \begin{pmatrix} 
  n \\ 
  \lceil \delta n\rceil 
\end{pmatrix} e^{-c_2 N/K^4} \le \exp\left( 4e\delta \log \left(\frac{e}{\delta}\right)n - \frac{c_2 N}{K^4}\right) \le \exp\left(-\frac{c_2 N}{2K^4}\right),\label{event_origin} 
\end{align}
with an appropriate choice of $c_1$.

Now we deduce the estimate for compressible vectors. Let $c_3>0$ and $\rho\in (0,1/2)$ to be chosen later. We need to control the event $W$ that $\|(PY-A)x\|_2 \le c_3 \sqrt{N}$ for some vector $x\in Comp(\delta,\rho)$. Assume $W$ occurs, then every such vector $x$ can be written as a sum $x=y+z$ with $y\in Sparse(\delta)$ and $\|z\|_2 \le \rho$. Thus $\|y\|_2 \ge 1- \rho \ge 1/2$, and
$$\|(PY-A)y\|_2 \le \|(PY-A)x\|_2 + \|(PY-A)\| \|z\|_2 \le c_3 \sqrt{N} + \rho C_1\sqrt{N}.$$
We choose $c_3 = 1/16$ and $\rho=1/(16C_1)$, so that $\|(PY-A)y\|_2 \le \sqrt{N}/8$. Since $\|y\|_2\ge 1/2$, we can find a unit vector $u= y/\|y\|_2 \in Sparse(\delta)$ such that $\|(PX-A)u\|_2 \le \sqrt{N}/4$. This shows that event $W$ implies the event in (\ref{event_origin}), so we have $\mathbb P(W) \le e^{-c_2 N/(2K^4)}$. This concludes the proof.
\end{proof}

\begin{remark}
If $n< c_1 N /(K^4\log K)$, then all the vectors in $S^{n-1}$ are in $Comp (\delta,\rho)$ and Lemma \ref{comp_bound} already concludes the proof of Theorem \ref{main_sub}. Hence throughout the following sections, it suffices to assume 
\begin{equation}\label{suff_largen}
n\ge c_1 N /(K^4\log K).
\end{equation}
\end{remark}

It remains to prove the bound for incompressible vectors in (\ref{incomp_bound}). Define the aspect ratio $\lambda:=n/N$. We will divide the proof into two cases: the case where $c_1 /(K^4\log K) \le \lambda \le \lambda_0$ for some constant $0<\lambda_0 < 1$, and the case where $\lambda_0 < \lambda \le 1$. 
We record here an important property of the incompressible vectors, which is proved in Lemma 3.4 of \cite{RudVersh_square}.
\begin{lem}[Incompressible vectors are spread]\label{lem_spread}
Let $x\in Incomp_n(\delta,\rho)$. Then there exists a set $\sigma \equiv \sigma(x) \subseteq \{1,\ldots, n\}$ of cardinality $|\sigma|\ge \frac{1}{2}\rho^2 \delta n$ and such that 
\begin{equation}\label{eq_spread}
\frac{\rho}{\sqrt{2n}} \le |x_k| \le \frac{1}{\sqrt{\delta n}} \ \ \text{ for all } k\in \sigma.
\end{equation}
\end{lem}

%

\section{Tall matrices}\label{sec_tall}

In this section, we deal with the probability in (\ref{incomp_bound}) when $c_1 /(K^4\log K) \le \lambda \le \lambda_0$ for some constant $\lambda_0\in (0,1)$. The value of $\lambda_0$ will be chosen later in Section \ref{sec_square} (see (\ref{lambda0})), and it only depends on $\Lambda$, $C_1$ and the maximal fourth moment of the entries of $Y$. Then it is equivalent to control the probability
$$ \mathbb P\left( \inf_{x\in Incomp_n(\delta,\rho)} \|(PY-A)x\|_2 \le t \sqrt{N} \right)$$
for any $t\ge 0$. 

Let $x$ be a vector in $Incomp_n(\delta,\rho)$, where we fix $\delta = c_1 N/\left(n K^{4}\log K\right)$ and $0 <\rho \le 1/(16C_1)$ (see Lemma \ref{comp_bound}). Take the set $\sigma$ given by Lemma \ref{lem_spread}. Note that the entries of $Yx$ are of the form $(Yx)_i=\sum_{k=1}^n Y_{ik}x_k$, $1\le i \le M$, where $Y_{ik}$ are independent centered random variables with unit variance and bounded fourth moment. Hence we can use Lemma \ref{lem_select} and Lemma \ref{small_CLT} to get that
\begin{equation}\label{small_entrywise}
\mathcal L\left((Yx)_i,t\right) \le \sqrt{\frac{2}{\pi}} \frac{t}{\|P_\sigma x\|_2} + C\left( \frac{\|P_\sigma x\|_3}{\|P_\sigma x\|_2}\right)^3 \le C_2(\rho) \left(\frac{t}{\sqrt{\delta}} + \frac{1}{\delta \sqrt{n}}\right),
\end{equation}
for some constant $C_2(\rho)>0$ depending only on $\rho$ and the maximal fourth moment. Here we used the bound 
$$\|P_\sigma x\|_2 \ge \frac{1}{2}\rho^2\sqrt{\delta}, \ \ \left( \frac{\|P_\sigma x\|_3}{\|P_\sigma x\|_2}\right)^3 \le \frac{2}{\rho^2 \delta \sqrt{ n}},$$
deduced from Lemma \ref{lem_spread}. With (\ref{small_entrywise}) as the input, the next lemma provides a small ball probability bound for the random vector $PYx$. 

\begin{lem}[Corollary 1.4 of \cite{RudVersh_smallball}]
Consider a random vector $X=(\xi_1, \ldots, \xi_M)$ where $\xi_i$ are real-valued independent random variables. Let $t,p\ge 0$ be such that
$$\mathcal L(\xi_i,t) \le p \ \text{ for all } i = 1,\ldots,M.$$
Let $P$ be an orthogonal projection in $\mathbb R^M$ onto an $N$-dimensional subspace. Then 
$$\mathcal L\left(PX, t\sqrt{N}\right) \le (Cp)^N,$$
where $C$ is an absolute constant.
\end{lem}

Applying the above lemma to random vector $Yx$, we obtain that 
\begin{equation}\label{small_PY}
\mathbb P\left(\|(PY-A)x\|_2 \le t \sqrt{N} \right)\le \mathcal L\left(PYx, t\sqrt{N}\right) \le \left[C_3 \left(\frac{t}{\sqrt{\delta}} + \frac{1}{\delta \sqrt{n}}\right)\right]^N
\end{equation}
for some constant $C_3>0$. Now we can take a union bound over all $x$ in an $\epsilon$-net of $Incomp_n(\delta,\rho)$ and complete the proof by approximation.

We first assume that $t\ge 1/\sqrt{\delta n}$. Then the $t/\sqrt{\delta}$ term in (\ref{small_PY}) dominates and we obtain that 
$$\mathbb P\left(\|(PY-A)x\|_2 \le t \sqrt{N} \right)\le \left(2C_3{t}/{\sqrt{\delta}}\right)^N.$$ 
By Lemma \ref{card_nets}, there exists an $\epsilon$-net $\mathcal N$ in $Incomp_n(\delta,\rho)$ of cardinality $|\mathcal N| \le 2n (5/\epsilon)^{n-1}$. Taking the union bound, we get
\begin{equation} \label{union2}
\mathbb P\left( \inf_{x\in \mathcal N} \|(PY-A)x\|_2 \le t\sqrt{N}\right) \le 2n\left(\frac{2C_3{t}}{\sqrt{\delta}}\right)^N \left(\frac{5}{\epsilon}\right)^{n-1}. 
\end{equation}
Let $V$ be the event that $\|(PY-A)y\|_2 \le t\sqrt{N}/2$ for some $y\in Incomp_n(\delta,\rho)$. Assume that $V$ occurs and choose a point $x\in \mathcal N$ such that $\|x-y\|_2 \le \epsilon$. Then if $\epsilon \le t/(2C_1)$, we have
$$\|(PY - A)x\|_2 \le \|(PY-A)y\|_2 + \|PY-A\|\|x-y\|_2 \le \frac{1}{2}t\sqrt{N} + C_1\epsilon \sqrt{N} \le t\sqrt{N},$$
where we used that $\|PY-A\|\le C_1 \sqrt{N}$. Fix one such $\epsilon$, using (\ref{union2}) we obtain that 
\begin{align}
& \mathbb P\left( \inf_{x\in Incomp_n(\delta,\rho)} \|(PY-A)x\|_2 \le \frac{t}{2} \sqrt{N}\right) = \mathbb P(V) \nonumber\\
& \le 2n\left(\frac{2C_3t}{\sqrt{\delta}}\right)^N \left(\frac{10C_1}{t}\right)^{n-1} \le \left[\left({C_4}{\delta^{-1/2}}\right)^{1/(1-\lambda_0)} t\right]^{N-n+1}, \label{tall_union} 
\end{align}
where in the last step we used $n/N\le \lambda_0$. If $t\le 1/\sqrt{\delta n}$, we use (\ref{tall_union}) to get
\begin{align*}
\mathbb P\left( \inf_{x\in Incomp_n(\delta,\rho)} \|(PY-A)x\|_2 \le \frac{t}{2} \sqrt{N}\right) \le \left[\left({C_4}{\delta^{-1/2}}\right)^{1/(1-\lambda_0)} (\delta n)^{-1/2}\right]^{N-n+1} \le e^{-c_0 N/K^4},
\end{align*}
if $K\le N^\omega$ for some sufficiently small $\omega$. Together with (\ref{tall_union}) and Lemma \ref{comp_bound}, this concludes the proof of Theorem \ref{main_sub} for the $\lambda\le \lambda_0$ case. 

\section{Almost square matrices}\label{sec_square} 

In this section, we deal with the probability in (\ref{incomp_bound}) for the $\lambda_0 < \lambda \le 1$ case. In particular, when $\lambda\to 1$, $PY-A$ becomes an almost square matrix and (\ref{tall_union}) cannot provide a satisfactory probability bound. For instance, for the square case with $N=n$, it is easy to see that the $(C\delta^{-1/2})^N$ term dominates over the $t$ term. To handle this difficulty, we will use the method in \cite{RudVersh_rect}, which reduces the problem of bounding $\|(PY-A)x\|_2$ for $x\in Incomp_n(\delta,\rho)$ to a random distance problem. We denote $N=n - 1 + d$ for some $d\ge 1$. Note that $\sqrt{N} - \sqrt{n-1} \le d/\sqrt{n}$. Hence to bound (\ref{incomp_bound}), it suffices to bound
\begin{equation}
\mathbb P\left( \inf_{x\in Incomp_n(\delta,\rho)} \|(PY-A)x\|_2 \le \epsilon\frac{d}{\sqrt{n}} \right), \ \ \text{for } \delta = c_1 N/\left(n K^{4}\log K\right), \ \rho \le 1/(16C_1).\label{def_deltarho}
\end{equation}

We denote
\begin{equation}\label{def_m}
m: = \min\left\{d, \left\lfloor \frac{1}{2}\rho^2\delta n \right\rfloor\right\}.
\end{equation}
Let $Z_1:=PY_1 - A_1,\ldots, Z_n:=PY_n - A_n$ be the columns of the matrix $Z:=PY-A$. Given a subset $J\subset \{1,\ldots, n\}$ of cardinality $m$, we define the subspace
\begin{equation}\label{def_HJc}
H_{J^c}:= \text{span}(Z_k)_{k\in J^c} \subseteq \mathbb R^N.
\end{equation}
For levels $K_1:= \rho\sqrt{\delta/2}$ and $K_2:=K_1^{-1}$, we define the set of totally spread vectors
\begin{equation}\label{total_spread}
S^J :=\left\{ y\in S^{n-1} \cap \mathbb R^J: \frac{K_1}{\sqrt{m}} \le |y_k| \le \frac{K_2}{\sqrt{m}} \text{ for all }k\in J\right\}.
\end{equation}
In the following lemma, we let $J$ be a random subset uniformly chosen over all subsets of $\{1,\ldots, n\}$ of cardinality $m$. We shall write $P_J$ for $P_{\mathbb R^J}$, the orthogonal projection onto the subspace $\mathbb R^J$. We denote the probability and expectation over the random subset $J$ by $\mathbb P_J$ and $\mathbb E_J$. 

\begin{lem}\label{lemma_probEx}
There exists constant $c_2>0$ depending only on $\rho$ such that for every $x\in Incomp_n(\delta,\rho)$, the event 
$$\mathcal E(x):=\left\{ \frac{P_J x}{\|P_J x\|_2} \in S^J \text{ and } \frac{\rho\sqrt{m}}{\sqrt{2n}} \le \|P_J x\|_2 \le \frac{\sqrt{m}}{\sqrt{\delta n}} \right\}$$
satisfies $\mathbb P_J(\mathcal E(x))\ge (c_2\delta)^m.$
\end{lem}
\begin{proof}
Let $\sigma \subset \{1, \ldots, n\}$ be the subset from Lemma \ref{lem_spread}. Then we have
$$\mathbb P_J \left( J \subset \sigma \right) = \begin{pmatrix} |\sigma|\\
m
\end{pmatrix}/ \begin{pmatrix} n \\
m
\end{pmatrix}.$$
Using Stirling's approximation, for $d\le \frac{1}{4}\rho^2 \delta n$, we have
$$\mathbb P_J \left( J \subset \sigma \right) \ge \left(\frac{c|\sigma|}{n} \right)^m \ge \left( c_2 \delta \right)^m,$$
and for $d> \frac{1}{4}\rho^2 \delta n $, we have
$$\mathbb P_J \left( J \subset \sigma \right) \ge \begin{pmatrix} n \\
m
\end{pmatrix}^{-1} \ge \frac{m!}{n^m} \ge \left( \frac{cm}{n} \right)^m \ge \left( c_2 \delta \right)^m.$$
If $J\subset \sigma$, then summing (\ref{eq_spread}) over $k\in J$, we obtain the required two-sided bound for $\|P_J x\|_2$. This and (\ref{eq_spread}) yield ${P_J x}/{\|P_J x\|_2} \in S^J $. Hence $\mathcal E(x)$ holds.
\end{proof}

Lemma \ref{lemma_probEx} implies the following lemma, whose proof is similar to the one for \cite[Lemma 6.2]{RudVersh_rect}.
\begin{lem}\label{incomp_red_dist}
Let $J$ denote the $m$-element subsets of $\{1,\ldots,n\}$. Then for every $\epsilon>0$,
\begin{equation}
\mathbb P\left( \inf_{x\in Incomp_n(\delta,\rho)} \|Zx\|_2 < \epsilon\rho \sqrt{\frac{m}{2n}}\right) \le \left( c_2 \delta\right)^{-m} \max_{J} \mathbb P\left( \inf_{x\in S^J} {\rm{dist}}(Zx,H_{J^c}) < \epsilon\right).
\end{equation}
\end{lem}

It remains to bound $\mathbb P\left( \inf_{x\in S^J} {\rm{dist}}(Zx,H_{J^c}) < \epsilon\right)$ for any $m$-element subset $J$. 
We shall need the following lemma to bound below the distance between a random vector in $\mathbb R^N$ and an independent random subspace of codimension $l$. It will be proved in Section \ref{section_distance}. 


\begin{lem}[Distance to a random subspace]\label{dist_lemm}
Let $J$ be any $m$-element subset of $\{1,\ldots,n\}$ and let $H_{J^c}$ be the random subspace of $\mathbb R^N$ defined in (\ref{def_HJc}). Let $X$ be a random vector in $\mathbb R^M$ whose coordinates are i.i.d. centered random variables with unit variance and finite fourth moments, independent of $H_{J^c}$. Assume that $l:=m+d-1 \le \beta N$. Then for every $\epsilon>0$, we have
\begin{equation}\label{dist_lemm_eq2}
\mathbb P\left( \sup_{v\in \mathbb R^N}\mathbb P \left( \left. {\rm{dist}}(PX-v,H_{J^c})< \epsilon \sqrt{l} \, \right | H_{J^c} \right) > (\tilde C\epsilon)^{l} + e^{-\tilde cN} \right) \le e^{-\tilde cN},
\end{equation}
where $\beta,\tilde c, \tilde C>0$ depend only on $\Lambda$, $C_1$ and the maximal fourth moment. 
\end{lem}
 
It is easy to see that (\ref{dist_lemm_eq2}) implies the weaker result:
\begin{equation}\label{dist_lemm_eq}
\sup_{v\in \mathbb R^N}\mathbb P \left( {\rm{dist}}(PX-v,H_{J^c})< \epsilon \sqrt{l}\right) \le (\tilde C\epsilon)^{l} + 2e^{-\tilde cN}.
\end{equation}
In the following proof, we choose $\lambda_0$ such that
\begin{equation}\label{lambda0}
d \le \beta N/2 \Rightarrow l \le 2d \le \beta N.
\end{equation}
Note that for any fixed $x\in S^J$, we have $Zx= PYx - Ax$, where $Yx$ is a random vector satisfying the assumptions for $X$ in Lemma \ref{dist_lemm}. 
So (\ref{dist_lemm_eq}) gives a useful probability bound for a single $x\in S^J$. Then we will try to take a union bound over all $x$ in an $\epsilon$-net
of $S^J$ and obtain a uniform distance bound. This is stated in the following theorem. 

\begin{thm}[Uniform distance bound]\label{uni_dist_bound}
Let $Y$ be a random matrix satisfying the assumptions in Theorem \ref{main_sub}. Then for every $m$-element subset $J$ and $t>0$, 
\begin{equation}\label{uni_dist_eq}
\mathbb P \left( \inf_{x\in S^J} {\rm{dist}}\left(Zx,H_{J^c}\right) < t\sqrt{d} \right) \le (\bar Ct K^5 \log K)^d + e^{-\bar cN},
\end{equation}
where $\bar C,\bar c>0$ depend only on $\tilde C$ and $\tilde c$.
\end{thm}

By the definition of $m$ in (\ref{def_m}), we have 
\begin{equation*}
\left( c_2 \delta\right)^{-m} \le \left[\left( c_2 \delta\right)^{-\rho^2\delta/2}\right]^{n} \le e^{\bar cN/2},
\end{equation*}
with an appropriate choice of $\rho$. Then we conclude from Lemma \ref{incomp_red_dist} and Theorem \ref{uni_dist_bound} that
\begin{align}
\mathbb P\left( \inf_{x\in Incomp_n(\delta,\rho)} \|Zx\|_2 < \epsilon \rho \sqrt{\frac{md}{2n}} \right) \le &(c_2 \delta )^{-m} (\bar C\epsilon K^5\log K )^d + e^{-\bar cN/2} \nonumber\\
\le & \left(CK^9(\log K)^2 \epsilon \right)^d + e^{-\bar cN/2}, \label{coro_of_dist}
\end{align}
where we used $m\le d$ and $\delta$ in (\ref{def_deltarho}). Changing $\epsilon$ to $\epsilon\rho^{-1}\sqrt{{2d}/{m}}$ in (\ref{coro_of_dist}) and using $d/m \le CK^4 \log K$, we get
\begin{align*}
\mathbb P\left( \inf_{x\in Incomp_n(\delta,\rho)} \|(PY-A)x\|_2 < \epsilon \frac{d}{\sqrt{n}} \right) \le  \left(CK^{11}(\log K)^{5/2} \epsilon \right)^d + e^{-\bar cN/2} ,
\end{align*}
which, together with Lemma \ref{comp_bound}, concludes the proof of Theorem \ref{main_sub}.

Now we begin the proof of Theorem \ref{uni_dist_bound}. Without loss of generality, we can assume that the entries of $Y$ have absolute continuous distributions. In fact we can add to each entry an independent Gaussian random variable with small variance $\sigma$, and later let $\sigma \to 0$ (all the estimates below do not depend on $\sigma$). Under this assumption, we have the following convenient fact:
\begin{equation}\label{reduction_dim}
\text{dim}(H_{J^c}) = n-m \ \ \text{a.s.}
\end{equation}
{


Let $P_{H^\perp}$ be the orthogonal projection in $\mathbb R^N$ onto $H^\perp_{J^c}$, and define
\begin{equation}\label{def_rm_W}
W:= \left.P_{H^\perp}PY\right|_{\mathbb R^J}.
\end{equation}
Then for every $x\in \mathbb R^n$, we have 
\begin{equation}\label{def_rm_w}
\text{dist}(PYx - v ,H_{J^c}) = \left\|Wx - w \right\|_2, \ \ \text{where } w= P_{H^\perp}v.
\end{equation}
By (\ref{reduction_dim}), $\text{dim}(H^\perp_{J^c})=N-n+m = l$ almost surely. Thus $W$ acts as an operator from an $m$-dimensional subspace into an $l$-dimensional subspace. If we have a proper operator bound for $W$, we can run the approximation argument on $S^J$ and prove a uniform distance bound over all $x\in S^J$.

\begin{prop}\label{opbound_W}
Let $W$ be a random matrix as in (\ref{def_rm_W}). Then 
$$\mathbb P \left( \left. \|W\| > sK\sqrt{d}\, \right| H_{J^c}\right) \le e^{- c_0 s^2 d}, \ \ \text{ for } s\ge C_0,$$
where $C_0,c_0>0$ are absolute constants.
\end{prop}
\begin{proof}
For simplicity of notations, we fix a realization of $H_{J^c}$ and omit the conditioning on it from the expressions below. Let $\mathcal N$ be an $(1/2)$-net of $S^{n-1}\cap \mathbb R^J$ and $\mathcal M$ be an $(1/2)$-net of $S^{n-1}\cap H^\perp_{J^c}$. By Lemma \ref{card_nets}, we can choose $\mathcal N$ and $\mathcal M$ such that
$$|\mathcal N| \le 5^{m}, \ \ |\mathcal M| \le 5^{l}.$$
It is easy to prove that
\begin{equation}\label{bound_Wnorm}
\|W\| \le 4\sup_{x\in \mathcal N, y \in \mathcal M} \left| \langle Wx,y\rangle \right|.
\end{equation}
For every $x\in \mathcal N$ and $y\in \mathcal M$, $\langle Wx,y\rangle = \langle PY x, y\rangle = \langle Yx, P^Ty\rangle$ is a random variable with subgaussian moment bounded by $CK$ for some absolute constant $C>0$. Hence by (\ref{assm_sub}) we have
$$\mathbb P \left(\left| \langle Wx,y\rangle \right| > \frac{1}{4}sK\sqrt{d}\right) \le 2e^{-cs^2 d}.$$
Using (\ref{bound_Wnorm}) and taking the union bound, we get that for large enough $C_0$,
$$\mathbb P \left(\|W\|> sK\sqrt{d}\right) \le 5^{m} \cdot 5^{l} \cdot 2e^{-cs^2 d} \le e^{-c_0 s^2 d}, \ \ \text{for } s\ge C_0,$$
where we used that $m\le l \le 2d$.
\end{proof}

\begin{lem}\label{approx_weak}
Let $W$ be a random matrix as in (\ref{def_rm_W}) and let $w$ be a random vector as in (\ref{def_rm_w}). Then for every $t\ge 0$, we have
\begin{equation}\label{probab_E1}
\mathbb P \left( \inf_{x\in S^J} \|Wx - w\|_2 < t\sqrt{d} , \|W\|\le C_0 K\sqrt{d}\right) \le K^{m-1}(C_2 t)^d + 2e^{- \tilde c N/4},
\end{equation}
where $C_2$ depends only on $\tilde C$.
\end{lem}
\begin{proof}
Fix any $x\in S^J$. It is easy to verify that $Yx$ is a random vector that satisfies the assumptions for $X$ in Lemma \ref{dist_lemm}. Hence by (\ref{def_rm_w}) and (\ref{dist_lemm_eq}), we have
\begin{equation}\label{lower_single}
\mathbb P \left( \|Wx - w\|_2 < t\sqrt{d} \right) \le \mathbb P \left( \text{dist}(PYx-v,H_{J^c})< t \sqrt{l}\right) \le (\tilde Ct)^{l} + 2e^{-\tilde c N}.
\end{equation}
Let $\epsilon = t/(C_0K)$. By Lemma \ref{card_nets}, there exists an $\epsilon$-net $\mathcal N$ of $S^J$ with $|\mathcal N|\le  2m({5 C_0 K}/{t})^{m-1}.$ Consider the event $$\mathcal E_t:=\left\{ \inf_{x\in \mathcal N} \left\| Wx - w \right\|_2 <2t\sqrt{d}\right\}.$$
Taking the union bound, we get that
\begin{align*} 
\mathbb P (\mathcal E_t ) \le 2m\left( \frac{5C_0 K}{t}\right)^{m-1}\left[(2\tilde Ct)^{m+d-1} + 2e^{-\tilde c N}\right]  \le K^{m-1} (C_2 t)^{d} + 4m\left( \frac{5C_0 K}{t}\right)^{m-1}e^{- \tilde c N}.
\end{align*}
For $t\ge t_0 := e^{-\tilde c N/(4d)}/(C_2 K) $, we have
$$4m\left( \frac{5C_0 K}{t}\right)^{m-1} \le \left({C_0' K^2}\right)^{\rho^2 \delta n /2} e^{\tilde cN/4} \le e^{\tilde c N/2}$$
with an appropriate choice of $\rho$. Thus we get
\begin{align*} 
\mathbb P (\mathcal E_t ) \le K^{m-1} (C_2 t)^{d} + e^{- \tilde c N/2}, \ \ \text{ for } t\ge t_0.
\end{align*}
For $t< t_0$, we have
\begin{align*} 
\mathbb P (\mathcal E_t ) \le \mathbb P (\mathcal E_{t_0} ) \le K^{m-1} (C_2 t_0)^{d} + e^{- \tilde c N/2} \le 2e^{- \tilde c N/4}.
\end{align*}
Then applying the standard approximation argument, we can check that the probability in (\ref{probab_E1}) is bounded by $\mathbb P (\mathcal E_t)$, which concludes the proof.
\end{proof}


With Proposition \ref{opbound_W} and Lemma \ref{approx_weak}, we obtain that
$$\mathbb P \left( \inf_{x\in S^J} \|Wx - w\|_2 < t\sqrt{d}\right) \le K^{m-1}(C_2 t)^d + 2e^{- \tilde c N/4} + e^{-c_0 C_0^2 d}.$$
Unfortunately, the bound $e^{-c_0 C_0^2 d}$ is too weak for small $d$. Following the idea in \cite{RudVersh_rect}, we refine the probability bound by decoupling the information about $\|Wx-w\|_2$ from the information about $\|W\|$. The proof of next lemma is essentially the same as the one for Proposition 7.5 of \cite{RudVersh_rect}. We omit the details.

\begin{lem}[Decoupling]\label{lem_decouple}
Let $X$ be an $N\times m$ matrix whose columns are independent random vectors, and let $A$ be an $N\times N$ deterministic matrix. Let $z\in S^{m-1}$ be a vector satisfying $|z_k| \ge K_1/\sqrt{m}$ for all $ k \in \{1,\ldots, m\}$. Then for every $v\in \mathbb R^N$ and every $0< a < b$, we have
\begin{equation*}
\mathbb P \left(\|AXz - Av\|_2 < a, \|AX\| > b \right) \le 2 \sup_{y\in S^{m-1}, u \in \mathbb R^N} \mathbb P \left(\|AXy - Au\|_2 < \frac{\sqrt{2}a}{K_1}\right)\mathbb P \left(\|AX\| > \frac{b}{\sqrt{2}} \right).
\end{equation*}
\end{lem}
\begin{remark}
By (\ref{total_spread}), all the vectors in $S^J$ satisfy the assumption for $z$ in Lemma \ref{lem_decouple}.
\end{remark}

With this decoupling lemma, we can prove the following refinement of Lemma \ref{approx_weak}.

\begin{lem}\label{approx_refine}
Let $W$ be a random matrix as in (\ref{def_rm_W}) and let $w$ be a random vector as in (\ref{def_rm_w}). For every $s\ge 1$ and every $t\ge 0$, we have
$$\mathbb P \left( \inf_{x\in S^J} \|Wx - w\|_2 < t\sqrt{d} \text{ and } sC_0 K \sqrt{d}< \|W\|\le 2sC_0K\sqrt{d}\right) \le \left[\frac{K^{m-1} \left(C_3 t\right)^d} {K_1^{m+d-1}} + 2e^{-\tilde c N/4}\right]e^{-c_1 s^2 d},$$
where $c_1$ is an absolute constant and $C_3$ depends only on $\tilde C$.
\end{lem}
\begin{proof}
Let $\epsilon = t/(2sC_0 K)$. By Lemma \ref{card_nets}, there exists an $\epsilon$-net $\mathcal N$ of $S^J$ with $|\mathcal N|\le 2m\left( {9sC_0 K}/{t}\right)^{m-1}.$
Consider the event $$\mathcal E_t :=\left\{ \inf_{x\in \mathcal N} \left\| Wx - w \right\|_2 < 2t\sqrt{d} \text{ and } \|W\| > sC_0 K\sqrt{d}\right\}.$$
Conditioning on $H_{J^c}$, we can apply Lemma \ref{lem_decouple} to get that 
\begin{align*}
\mathbb P\left(\left.\mathcal E_t \right|H_{J^c}\right) \le |\mathcal N| \cdot 2 \sup_{x\in S^{m-1}, v \in \mathbb R^N} \mathbb P \left( \left. \|Wx - P_{H^\perp}v\|_2 < \frac{\sqrt{2}}{K_1} \cdot 2t\sqrt{d} \, \right| H_{J^c}\right)\mathbb P \left(\left. \|W\|\ge \frac{sC_0K\sqrt{d}}{\sqrt{2}}\,\right| H_{J^c} \right)
\end{align*}
Taking expectation over $H_{J^c}$ and using Proposition \ref{opbound_W}, we obtain that
\begin{align*}
\mathbb P(\mathcal E_t) & \le 4m\left( \frac{9sC_0 K}{t}\right)^{m-1}e^{-c_0 C_0^2 s^2 d/2} \mathbb E \left[ \sup_{x\in S^{m -1}, v \in \mathbb R^N} \mathbb P \left( \left. \|Wx - P_{H^\perp}v\|_2 < \frac{\sqrt{2}}{K_1} \cdot 2t\sqrt{d} \, \right| H_{J^c}\right)\right] \nonumber\\
& \le 4m\left( \frac{9C_0 K}{t}\right)^{m-1} \left(s^{m-1}e^{-c_0 C_0^2 s^2 d/2}\right) \left[ \left(\frac{2\sqrt{2}\tilde Ct}{K_1}\right)^{m+d-1} + 2 e^{- \tilde cN}\right],
\end{align*}
where in the second step we used the representation in (\ref{def_rm_w}) and the estimate (\ref{dist_lemm_eq2}). Since $s\ge 1$ and $1\le m \le d$, we can bound this as
\begin{align*}
\mathbb P(\mathcal E_t) \le \left[\frac{K^{m-1} \left(C_3 t\right)^d} {K_1^{m+d-1}} + C_4 m\left( \frac{9C_0 K}{t}\right)^{m-1} e^{-\tilde c N}\right]e^{-c_1 s^2 d},
\end{align*}
where $C_4>0$ is an absolute constant.
For $t\ge t_1 := e^{-\tilde c N/(4d)}K_1^2 /(C_3 K)$, we have
$$C_4 m\left( \frac{9C_0 K}{t}\right)^{m-1} \le \left({C_0' K^2 K_1^{-2}}\right)^{\rho^2 \delta n /2}  e^{\tilde c N/4} \le e^{\tilde c N/2}$$
with an appropriate choice of $\rho$. Thus we get
\begin{align*} 
\mathbb P (\mathcal E_t ) \le \left[\frac{K^{m-1} \left(C_3 t\right)^d} {K_1^{m+d-1}} + e^{- \tilde c N/2}\right]e^{-c_1 s^2 d}, \ \ \text{ for } t\ge t_1.
\end{align*}
For $t< t_1$, we have
\begin{align*} 
\mathbb P (\mathcal E_t ) \le \mathbb P (\mathcal E_{t_1} ) \le \left[\frac{K^{m-1} \left(C_3 t_1\right)^d} {K_1^{m+d-1}} + e^{- \tilde c N/2}\right]e^{-c_1 s^2 d} \le 2e^{- \tilde c N/4}e^{-c_1 s^2 d}.
\end{align*}
Suppose there exists $y\in S^J$ such that 
$$ \|Wy - w\|_2 < t\sqrt{d} \ \text{ and }\ sC_0 K \sqrt{d}< \|W\|\le 2sC_0K\sqrt{d}.$$
Then we choose $x\in \mathcal N$ such that $\|x-y\|_2 \le \epsilon$, and by triangle inequality we obtain that
$$\|Wx - w\|_2 \le \|Wy - w\|_2 + \|W\|\|x-y\|_2 < t\sqrt{d} + 2sC_0K\sqrt{d} \epsilon \le 2t\sqrt{d},$$
i.e. the event $\mathcal E_t$ holds. Then the bound for $P(\mathcal E_t)$ concludes the proof. \end{proof}
 
\begin{proof} [Proof of Theorem \ref{uni_dist_bound}]
Summing the probability bounds in Lemma \ref{approx_weak} and Lemma \ref{approx_refine} for $s=2^k$, $k \in \mathbb Z_+$, we conclude that 
\begin{align*}
\mathbb P \left( \inf_{x\in S^J} \|Wx - w\|_2 < t\sqrt{d} \right) & \le K^{m-1}(C_2 t)^d + 2e^{- \tilde c N/4} + \left[\frac{K^{m-1} \left(C_3 t\right)^d} {K_1^{m+d-1}} + 2e^{-\tilde c N/4}\right] \sum_{s=2^k, k \in \mathbb Z_+}e^{-c_1 s^2 d} \\
&\le \left( C_5 KK_1^{-2} t \right)^d + C_6 e^{- \tilde c N/4}.
\end{align*}
Using that $K_1 = \rho\sqrt{\delta/2}$ (see (\ref{total_spread})) and $\delta= c_1 N/(nK^4 \log K)$ (see (\ref{def_deltarho})), we get
\begin{align*}
\mathbb P \left( \inf_{x\in S^J} \|Wx - w\|_2 < t\sqrt{d} \right) \le \left( C t K^5\log K \right)^d + Ce^{- \tilde c N/4}.
\end{align*}
In view of the representation (\ref{def_rm_w}), this concludes the proof. 
\end{proof}

\section{Proof of Lemma \ref{dist_lemm}}\label{section_distance}

We will first prove a general inequality that holds for any fixed subspace $H$ in $\mathbb R^N$ of codimension $l= m+d-1$. This probability bound will depend on the arithmetic structure of $H$, which can be expressed using the {\it{least common denominator}} (LCD). 
Following the notations in \cite{RudVersh_rect}, for $\alpha>0$ and $\gamma\in (0,1)$, we define the least common denominator of a vector $a \in \mathbb R^M$ as
$${\rm{LCD}}_{\alpha,\gamma}(a):=\inf\left\{ \theta>0: {\rm{dist}}(\theta a,\mathbb Z^M) < \min(\gamma \|\theta a\|_2,\alpha)\right\}.$$ 
More generally, let $a=(a_1, \ldots, a_M)$ be a sequence of vectors $a_k \in \mathbb R^l$. We define the product of such multi-vector $a$ and a vector $\theta\in \mathbb R^l$ as
$$\theta \cdot a := \left( \langle \theta , a_1\rangle , \ldots, \langle \theta , a_M\rangle\right)\in \mathbb R^M.$$
Then we define, for $\alpha>0$ and $\gamma \in (0,1)$,
$${\rm{LCD}}_{\alpha,\gamma}(a):=\inf\left\{ \|\theta\|_2: \theta\in \mathbb R^l, {\rm{dist}}(\theta \cdot a,\mathbb Z^M) < \min(\gamma \|\theta \cdot a\|_2,\alpha)\right\}.$$
Finally, the least common denominator of a subspace $E\subseteq \mathbb R^M$ is defined as
\begin{equation}\label{LCD_space}
\text{LCD}_{\alpha,\gamma}(E):=\inf\{\text{LCD}_{\alpha,\gamma}(a): a \in S(E)\} =\inf\left\{\|\theta\|_2:\theta\in E, \text{dist}(\theta,\mathbb Z^M)< \min(\gamma \|\theta\|_2,\alpha)\right\}.
\end{equation}
A key to the proof is the next small ball probability theorem.
\begin{thm}[Theorem 3.3 of \cite{RudVersh_rect}] 
\label{thm_smallball}
Consider a sequence $a=(a_1, \ldots, a_M)$ of vectors $a_k \in \mathbb R^l$, which satisfies
\begin{equation}\label{isotropic_x}
\sum_{k=1}^M \langle x, a_k\rangle^2 \ge \|x\|_2^2 \ \ \text{ for every } x\in \mathbb R^l.
\end{equation}
Let $\xi_1,\ldots,\xi_M$ be i.i.d. centered random variables, such that $\mathcal L(\xi_k, 1) \le 1-b$ for some $b>0$. Consider the random sum $S:=\sum_{k=1}^M a_k \xi_k \in \mathbb R^l$. Then, for every $\alpha >0$ and $\gamma \in (0,1)$, and for 
$$\epsilon \ge \frac{\sqrt{l}}{{\rm{LCD}}_{\alpha, \gamma} (a)},$$
we have
$$\mathcal L\left(S,\epsilon \sqrt{l}\right) \le \left(\frac{C\epsilon}{\gamma \sqrt{b}}\right)^l + C^l e^{-2b\alpha^2}.$$
\end{thm}


Let $H$ be a fixed subspace in $\mathbb R^N$ of codimension $l$. We denote an orthonormal basis of $H^\perp$ by $\{n_1 , \ldots, n_l\} \subseteq \mathbb R^N$, and write $X$ in coordinates as $X=(\xi_1 , \ldots, \xi_M)$. Then using $PP^T=1$, we get
\begin{align*}
\text{dist}(PX-v,H) & = \|P_{H^\perp} (PX - v)\|_2 = \left\| \sum_{r = 1}^l \langle PX, n_r\rangle n_r - P_{H^\perp} v\right\|_2 = \left\| \sum_{r = 1}^l \langle X ,P^T n_r\rangle n_r - P_{H^\perp} v\right\|_2 \\
& = \left\| \sum_{r = 1}^l \langle X ,P^T n_r\rangle P^T n_r - P^T P_{H^\perp} v\right\|_2 = \|P_E X - w\|_2 =  \left\| \sum_{k=1}^M a_k \xi_k - w\right\|_2,
\end{align*}
where 
$$E\equiv E(H):=P^T H^\perp, \ \ a_k :=P_E e_k, \ \ w := P^TP_{H^\perp}v ,$$
and where $e_1, \ldots, e_M$ denote the canonical basis of $\mathbb R^M$. Notice that
$$\sum_{k=1}^M \langle x , a_k \rangle^2 = \|x\|_2^2, \ \ \text{for any }x\in E.$$
Hence we can use Theorem \ref{thm_smallball} in the space $E$ (identified with $\mathbb R^l$ by a suitable isometry). For every $\theta=(\theta_1, \ldots, \theta_M)\in E$ and every $k$, we have
$\langle \theta, a_k\rangle = \langle \theta, e_k\rangle=\theta_k$, so $\theta \cdot a = \theta$, where the right hand side is considered as a vector in $\mathbb R^M$. Therefore, we have
$$\text{LCD}_{\alpha,\gamma}(E)= \text{LCD}_{\alpha,\gamma}(a).$$
By Lemma \ref{lem_Paley}, $\mathcal L(\xi_k,1/2) \le 1-b$ for some $b>0$ that depends only on the fourth moment of $\xi_k$. Hence we can apply Theorem \ref{thm_smallball} to $S=\sum_{k=1}^M a_k \xi_k$ and conclude that for every $\epsilon >0$,
\begin{equation}\label{reduce_to_LCD}
\mathbb P\left(\text{dist}(PX-v,H)<\epsilon\sqrt{l}\right) \le \mathcal L(S,\epsilon \sqrt{l}) \le \left(\frac{C\epsilon}{\gamma}\right)^l + \left(\frac{C\sqrt{l}}{\gamma{\rm{LCD}}_{\alpha,\gamma}(E)}\right)^l + C^l e^{-c\alpha^2}.
\end{equation}
Now it suffices to bound below the least common denominator of the random subspace $E$. Heuristically, the randomness should remove any arithmetic structure from the subspace $E$ and make the LCD exponentially large. The next theorem shows that this is indeed true. 

\begin{thm}\label{structure_thm}
Suppose $\xi_1, \ldots , \xi_{N-l}$ are independent centered random variables with unit variance and uniformly bounded fourth moment. Let $\tilde Y$ be an $M\times (N-l)$ random matrix whose rows are independent copies of the random vector $(\xi_1,\ldots,\xi_{N-l})$, and $\tilde A$ be an $N\times (N-l)$ deterministic matrix. Suppose that $\|\tilde Y\| + \|\tilde A\|\le C_1\sqrt{N}$ for some constant $C_1>0$. Let $H$ be the random subspace of $\mathbb R^N$ spanned by the column vectors of $P\tilde Y-\tilde A$, and define the subspace $E\equiv E(H) : = P^T H^\perp \subseteq \mathbb R^M$. Then for $\alpha= c\sqrt{N}$, we have
$$\mathbb P\left({\rm{LCD}}_{\alpha,c}(E) < c\sqrt{N}e^{cN/l}\right) \le e^{-cN},$$
where $c$ depends only on $\Lambda$, $C_1$ and the maximal fourth moment.
\end{thm}

\begin{proof}[Proof of Lemma \ref{dist_lemm}]
Consider the event $\mathcal E:= \left\{ {\rm{LCD}}_{\alpha,c}\left(E(H_{J^c})\right) \ge c\sqrt{N}e^{cN/l} \right\}$. The above theorem shows that $\mathbb P(\mathcal E) \ge 1- e^{-cN} $. Conditioning on a realization of $H_{J^c}$ in $\mathcal E$, we obtain from (\ref{reduce_to_LCD}) that
\begin{equation}
\sup_{v\in \mathbb R^N} \mathbb P\left(\left. \text{dist}(PX-v,H_{J^c})<\epsilon\sqrt{l} \right| H_{J^c} \right) \le \left({C' \epsilon}\right)^l + \left(C' \right)^l e^{-c' N}, \ \ \text{for } H_{J^c}\in \mathcal E.
\end{equation}
Since $l\le \beta N$, with an appropriate choice of $\beta$ we get 
\begin{equation*}
(C')^l \le e^{c'N/2}.
\end{equation*}
Then the proof is completed by the estimate on the probability of $\mathcal E^c$.
\end{proof}

The rest of this section is devoted to proving Theorem \ref{structure_thm}. Note that if $a\in E(H)$, then $a= P^T b$ for some $b\in H^\perp$. Then with $b= Pa$, we have that
\begin{equation*}
b\in H^\perp \Leftrightarrow \tilde Y^T P^T b - \tilde A^T b = 0 \Leftrightarrow \tilde Y^T a - \tilde A^T P a = 0.
\end{equation*}
We denote $\tilde B := \tilde A^T P$. 
For every set $S$ in $E$, we have
\begin{equation}\label{navg_E}
\inf_{x\in S} \left\|\tilde Y^T x - \tilde Bx\right\|_2 >0 \text{ implies } S \cap E=\emptyset.
\end{equation}
This helps us to ``navigate" the random subspace $E$ away from undesired sets $S$ on the unit sphere. 

As in Definition \ref{defn_compress}, we can define the compressible and incompressible vectors on $S^{M-1}$, which are denoted by $Comp_M (\delta,\rho)$ and $Incomp_M(\delta,\rho)$, respectively. First, we have the following result for compressible vectors. 

\begin{lem}[Random subspaces are incompressible]\label{random_not_comp}
There exist $\delta,\rho\in (0,1)$ such that
\begin{equation}\label{compress_space}
\mathbb P\left(E\cap Comp_M({\delta,\rho}) = \emptyset \right) \ge 1- e^{-c_0 N},
\end{equation}
where the constants $\delta, \rho, c_0>0$ depend only on $\Lambda$, $C_1$ and the maximal fourth moment.
\end{lem}
\begin{proof}
Due to (\ref{navg_E}), it suffices to prove that
\begin{equation}\label{compress_lower}
\mathbb P\left( \inf_{x\in Comp_M(\delta,\rho)} \left\|\left(\tilde Y^T - \tilde B\right)x\right\|_2 \le c_0 \sqrt{N}\right) \le e^{-c_0 N}.
\end{equation}
In fact, the proof is similar to the one for Lemma \ref{comp_bound}. However, instead of Lemma \ref{lem_smallball}, we will use the fact that $\tilde Y^T$ has independent row vectors $\tilde Y_1, \ldots, \tilde Y_{N-l}$. For any $x\in S^{M-1}$, it is easy to verify that $\langle \tilde Y_k, x\rangle$ has variance 1 and uniformly bounded fourth moment. Then by Lemma \ref{lem_Paley}, there exists a $p\in (0,1)$ such that for any fixed $v=(v_1, \ldots, v_{N-l})\in \mathbb R^{N-l}$,
$$\mathbb P \left( |\langle \tilde Y_k, x\rangle - v_k| \le 1/2 \right) \le p.$$
By Lemma \ref{lem_tensor}, we can find constants $\eta,\nu \in (0,1)$ depending on $p$ only and such that 
\begin{equation}\label{estimate_iid}
\mathbb P\left\{ \|\tilde Y^Tx - v\|_2 \le \eta\sqrt{N-l}\right\} \le \nu^{N-l}.
\end{equation}
Recall that $l\le \beta N$ and $M\le \Lambda N$ by our assumptions. Then using (\ref{estimate_iid}) instead of (\ref{incomp_key}), we can complete the proof of (\ref{compress_lower}) as in Lemma \ref{comp_bound}. 
\end{proof}

Fix the constants $\delta$ and $\rho$ given by Lemma \ref{random_not_comp} for the rest of this section. Note that in contrast to the case in Lemma \ref{comp_bound}, $\delta$ is now an $N$-independent constant. We will further decompose $Incomp_M(\delta,\rho)$ into level sets $S_D$ according to the value $D$ of the LCD. We shall prove a nontrivial lower bound on $\inf_{x\in S_D} \|(\tilde Y^T - \tilde B)x\|_2$ for each level set up to $D$ of the exponential order. By (\ref{navg_E}), this means that $E$ is disjoint from every such level set. Therefore, $E$ must have exponentially large LCD. First, as a consequences of Lemma \ref{lem_spread}, we have the following lemma, which gives a weak lower bound for the LCD. 
\begin{lem}[Lemma 3.6 of \cite{RudVersh_rect}]\label{priori_LCD}
For every $\delta,\rho\in (0,1)$, there exist $c_1(\delta,\rho)>0$ and $c_2(\delta)>0$ such that the following holds. Let $a\in Incomp_M(\delta,\rho)$. Then for every $0< c <c_1(\delta,\rho)$ and every $\alpha>0$, one has
$${\rm{LCD}}_{\alpha,c}(a)>c_2(\delta)\sqrt{M}.$$
\end{lem}



\begin{defn}[Level sets]\label{level_sets}
Let $D\ge c_2(\delta) \sqrt{M}$. Define $S_D\subseteq S^{M-1}$ as
$$S_D:= \{x\in Incomp_M(\delta,\rho): D\le {\rm{LCD}}_{\alpha,c}(x) < 2D\}\cap \left(P^T \mathbb R^N\right).$$
\end{defn}

To obtain a lower bound for $\|(\tilde Y^T - \tilde B)x\|_2$ on $S_D$, we use the $\epsilon$-net argument again. We first need such a bound for a single vector $x$. The proof of next lemma is very similar to the one for Lemma 4.6 in \cite{RudVersh_rect}. We omit the details.

\begin{lem}\label{structure_single}
Let $x\in S_D$. Then for every $t>0$ we have
\begin{equation}\label{lem_dom_t}
\mathbb P\left(\|(\tilde Y^T  - \tilde B) x\|_2 < t\sqrt{N}\right) \le \left( Ct + \frac{C}{D} + Ce^{-c\alpha^2}\right)^{N-l}.
\end{equation}
\end{lem}

Now we construct a small $\epsilon$-net of $S_D$. 
Our argument here is a little harder than the one in \cite{RudVersh_rect}, because the $\epsilon$-net lies in a subspace $P^T \mathbb R^N \subseteq \mathbb R^M$, whose direction is quite arbitrary. We shall need the following classical result in geometric functional analysis \cite{Ball}.

\begin{lem}\label{ball}
If $S\subseteq \mathbb R^M$ is a subspace of codimension $k$, then
\[\left|S\cap Q_M\right| \le (\sqrt{2})^{k},\]
where $Q_M = [-1/2,1/2]^M$ is the unit cube centered at the origin.
\end{lem}

\begin{lem}\label{net_counting}
There exists a $(4\alpha/D)$-net of $S_D$ of cardinality at most $(CD/\sqrt{N})^N$.
\end{lem}

\begin{proof}
We can assume that $4\alpha/D\le 1$, otherwise the conclusion is trivial. For $x\in S_D$, we denote $D(x):={\rm{LCD}}_{\alpha,c}(x).$ By the definition of $S_D$, we have $D\le D(x) < 2D$. By the definition of LCD, there exists $p\in \mathbb Z^M$ such that
\begin{equation}\label{NoLS1}
\|D(x)x-p\|_2 < \alpha.
\end{equation}
Therefore,
$$\left\|x- \frac{p}{D(x)}\right\| < \frac{\alpha}{D(x)} \le \frac{1}{4}.$$
Since $\|x\|_2 = 1$, it follows that 
\begin{equation*}
\left\| x - \frac{p}{\|p\|_2}\right\|_2 \le \frac{2\alpha}{D}.
\end{equation*}
We can chose $p$ such that it is the closest integer point to $D(x)x$. Since $\|D(x)x\|_2< 2D$, $p$ must lie in the ``cube covering" $\tilde F$ of $F:=B(0,2D)\cap P^T \mathbb R^N$, defined as
\begin{equation*}
\tilde F := \bigcup\limits_{b\in F}\left(\prod\limits_{i=1}^M[b_i-1/2, b_i+1/2]\right).
\end{equation*}
On the other hand, by (\ref{NoLS1}) and using that $\|D(x)x\|_2< 2D$ and $4\alpha/D\le 1$, we obtain
\begin{equation*}
\|p\|_2<D(x)+\alpha\le 3D.
\end{equation*}
In sum, we get a $(2\alpha/D)$-net of $S_D$ as:
$$\mathcal N := \left\{\frac{p}{\|p\|_2}:p\in\mathbb Z^M\cap B(0, 3D)\cap\tilde F\right\}.$$
The cardinality of $\mathcal N$ can be bounded by the volume of $B(0, 3D)\cap\tilde F$. 
By Fubini's theorem, we have
$$\left|B(0, 3D)\cap\tilde F\right| \le \left|B(0,3D)\cap S\right| \cdot \left|S^\perp\cap Q_M\right|, \ \ S:=P^T \mathbb R^N.$$
Then using the volume formula for an $N$-dimension ball and Lemma \ref{ball}, we obtain that
\[|\mathcal N| \le (CD/\sqrt{N})^N.\]
Finally, we can find a $4\alpha/D$-net of the same cardinality, which lies in $S_D$ (see Lemma 5.7 of \cite{RudVersh_square}). This completes the proof.
\end{proof}

\begin{lem}\label{lower_bounds}
There exist $c_3, c_4, \mu\in(0,1)$ such that the following holds. Let $\alpha=\mu\sqrt{N}\ge 1$ and $D\le c_3\sqrt N e^{c_3N/l}$. Then
$$\mathbb P \left( \inf_{x\in S_D} \left\|\left(\tilde Y^T - \tilde B\right) x\right\|_2 < c_4 N/D\right)\le e^{-N}.$$
\end{lem}

\begin{proof}
To conclude the proof, it is enough to find $\nu >0$ such that the event 
$$\mathcal E:=\left\{ \inf_{x\in S_D} \left\|\left(\tilde Y^T - \tilde B\right)x\right\|_2 < \frac{\nu N}{2D}\right\}$$
has probability $\le e^{-N}$. Let $\nu >0$ be a small constant to be chosen later. We apply Lemma \ref{structure_single} with $t=\nu\sqrt{N}/D$. By the assumptions on $\alpha$ and $D$, the term $Ct$ dominates in the right hand side of (\ref{lem_dom_t}). This gives for arbitrary $x \in S_D$,
$$\mathbb P\left(\left\|\left(\tilde Y^T - \tilde B\right)x\right\|_2 < \frac{\nu N}{D}\right) \le \left( \frac{C\nu \sqrt{N}}{D}\right)^{N-l}.$$
We take the $(4\alpha/D)$-net $\mathcal N$ of $S_D$ given by Lemma \ref{net_counting}, and take the union bound to get
$$p:=\mathbb P \left( \inf_{x\in \mathcal N} \left\|\left(Y^T - B\right)x \right\|_2 < \frac{\nu N}{D}\right)\le \left(\frac{CD}{\sqrt{N}}\right)^N \left( \frac{C\nu \sqrt{N}}{D}\right)^{N-l} \le \left(\frac{CD}{\sqrt{N}}\right)^l \left( C'\nu\right)^{N-l} .$$
Using the assumption on $D$, we can choose $\nu$ small enough such that
$$p \le \left(C''\right)^l e^{c_3N} \left( C'\nu\right)^{N-l} \le e^{-N},$$
where we used $l \le \beta N$ in the last step.

Now assume $\mathcal E$ holds. By the assumption of Theorem \ref{structure_thm}, we have 
$$\|\tilde Y^T - \tilde B\| \le \|\tilde Y\| + \|\tilde A\| \le C_1\sqrt{N}.$$ 
Fix $x\in S_D$ such that $\|(\tilde Y^T - \tilde B)x\| < \nu N/(2D)$. Then we can find $y\in \mathcal N$ such that $\|x-y\| \le {4\alpha}/{D}.$ Then, by the triangle inequality we have
\begin{align*}
\left\|\left(\tilde Y^T - \tilde B\right)y\right\|_2 & \le \left\|\left(\tilde Y^T -\tilde B\right)x\right\|_2 + \left\|\tilde Y^T-\tilde B\right\|\cdot \|x-y\|_2  \le \frac{\nu N}{2D} + C_1 \sqrt{N}\frac{4\mu\sqrt{N}}{D} <\frac{\nu N}{D},
\end{align*}
if we choose $\mu < \nu /(8C_1)$. Thus we get
$$\mathbb P(\mathcal E) \le \mathbb P \left( \inf_{x\in \mathcal N} \left\|\left(Y^T - B\right)x \right\|_2 < \frac{\nu N}{D}\right)\le e^{-N},$$
which concludes the proof.
\end{proof}

\begin{proof}[Proof of Theorem \ref{structure_thm}]
Consider $x\in S^{M-1}\cap E$ such that
$${\rm{LCD}}_{\alpha, c}(x) < c_3 \sqrt N e^{c_3N/l}.$$
Then, by Lemma \ref{priori_LCD} and Definition \ref{level_sets}, either $x$ is compressible or $x\in S_D$ for some $D\in\mathcal D$, where 
$$\mathcal D:=\left\{D:c_2\sqrt N \le D<c_3\sqrt{N} e^{c_3 N/l}, D=2^k, k\in\mathbb N \right\},$$
where we used that $M\ge N$. Therefore, we can decompose the desired probability as follows:
\begin{align*}
p:=\mathbb P\left({\rm{LCD}}_{\alpha, c}(E)<c_3 \sqrt N e^{c_3 N/l}\right) \le \mathbb P\left(E\cap Comp_M({\delta,\rho}) \ne \emptyset \right) + \sum_{D\in\mathcal D}\mathbb P(E\cap S_D\ne \emptyset).
\end{align*}
The first term can be bounded by $e^{-c_0 N}$ by Lemma \ref{random_not_comp}. 
The other terms can be bounded with (\ref{navg_E}) and Lemma \ref{lower_bounds}:
\[\mathbb P(E\cap S_D\ne \emptyset)\le \mathbb P\left(\inf_{x\in S_D} \left\|\left(\tilde Y^T - \tilde B\right)x\right\|_2 =0\right)\le e^{-N}.\]
Since there are $|\mathcal D|\le CN$ terms in the sum, we conclude that
\[p\le e^{-c_0 N}+CNe^{-N}\le e^{-c'N}.\]
This concludes the proof.
\end{proof}


\bibliographystyle{abbrv}
\bibliography{band_bib}

\begin{thebibliography}{10}

\bibitem{Bai1997}
Z.~D. Bai.
\newblock Circular law.
\newblock {\em Ann. Probab.}, 25(1):494--529, 1997.

\bibitem{BaiYin_law}
Z.~D. Bai and Y.~Q. Yin.
\newblock Limit of the smallest eigenvalue of a large dimensional sample
  covariance matrix.
\newblock {\em Ann. Probab.}, 21(3):1275--1294, 1993.

\bibitem{Ball}
K.~Ball.
\newblock {\em Volumes of sections of cubes and related problems}, pages
  251--260.
\newblock Springer Berlin Heidelberg, Berlin, Heidelberg, 1989.

\bibitem{isotropic}
A.~Bloemendal, L.~Erd{\H o}s, A.~Knowles, H.-T. Yau, and J.~Yin.
\newblock Isotropic local laws for sample covariance and generalized {W}igner
  matrices.
\newblock {\em Electron. J. Probab.}, 19(33):1--53, 2014.

\bibitem{Handbook_DS}
K.~R. Davidson and S.~J. Szarek.
\newblock Local operator theory, random matrices and banach spaces.
\newblock volume~1 of {\em Handbook of the Geometry of Banach Spaces}, pages
  317 -- 366. North-Holland, Amsterdam, 2001.

\bibitem{Geman_large}
S.~Geman.
\newblock A limit theorem for the norm of random matrices.
\newblock {\em Ann. Probab.}, 8(2):252--261, 1980.

\bibitem{Ginibre}
J.~Ginibre.
\newblock Statistical ensembles of complex, quaternion, and real matrices.
\newblock {\em J. Math. Phys.}, 6(3):440--449, 1965.

\bibitem{gotze2010}
F.~G{\H o}tze and A.~Tikhomirov.
\newblock The circular law for random matrices.
\newblock {\em Ann. Probab.}, 38(4):1444--1491, 2010.

\bibitem{Anisotropic}
A.~Knowles and J.~Yin.
\newblock Anisotropic local laws for random matrices.
\newblock {\em Probability Theory and Related Fields}, pages 1--96, 2016.

\bibitem{Random_polytopes}
A.~Litvak, A.~Pajor, M.~Rudelson, and N.~Tomczak-Jaegermann.
\newblock Smallest singular value of random matrices and geometry of random
  polytopes.
\newblock {\em Adv. Math.}, 195(2):491 -- 523, 2005.

\bibitem{Rud_polytope}
A.~Litvak, A.~Pajor, M.~Rudelson, and N.~Tomczak-Jaegermann.
\newblock Smallest singular value of random matrices and geometry of random
  polytopes.
\newblock {\em Advances in Mathematics}, 195(2):491 -- 523, 2005.

\bibitem{MP}
V.~A. Mar{\v c}enko and L.~A. Pastur.
\newblock Distribution of eigenvalues for some sets of random matrices.
\newblock {\em Mathematics of the USSR-Sbornik}, 1:457, 1967.

\bibitem{PanZhou_circular}
G.~Pan and W.~Zhou.
\newblock Circular law, extreme singular values and potential theory.
\newblock {\em J. Multivar. Anal.}, 101(3):645--656, 2010.

\bibitem{Rud_Annal}
M.~Rudelson.
\newblock Invertibility of random matrices: norm of the inverse.
\newblock {\em Ann. Math.}, 168(2):575--600, 2008.

\bibitem{RudVersh_square}
M.~Rudelson and R.~Vershynin.
\newblock The {L}ittlewood-{O}fford problem and invertibility of random
  matrices.
\newblock {\em Adv. Math.}, 218:600--633, 2008.

\bibitem{RudVersh_rect}
M.~Rudelson and R.~Vershynin.
\newblock The smallest singular value of a random rectangular matrix.
\newblock {\em Comm. Pure Appl. Math.}, 62:1707--1739, 2009.

\bibitem{HansonW}
M.~Rudelson and R.~Vershynin.
\newblock {H}anson-{W}right inequality and sub-gaussian concentration.
\newblock {\em Electron. Commun. Probab.}, 18:9 pp., 2013.

\bibitem{RudVersh_smallball}
M.~Rudelson and R.~Vershynin.
\newblock Small ball probabilities for linear images of high-dimensional
  distributions.
\newblock {\em Int. Math. Res. Notices}, 2015(19):9594--9617, 2015.

\bibitem{Silverstein_small}
J.~W. Silverstein.
\newblock The smallest eigenvalue of a large dimensional {W}ishart matrix.
\newblock {\em Ann. Probab.}, 13(4):1364--1368, 1985.

\bibitem{TaoVu_circular}
T.~Tao and V.~Vu.
\newblock Random matrices: the circular law.
\newblock {\em Commun. Contemp. Math.}, 10(2):261--307, 2008.

\bibitem{tao2010}
T.~Tao, V.~Vu, and M.~Krishnapur.
\newblock Random matrices: Universality of {ESD}s and the circular law.
\newblock {\em Ann. Probab.}, 38(5):2023--2065, 2010.

\bibitem{XYY}
H.~Xi, F.~Yang, and J.~Yin.
\newblock Local circular law for the product of a deterministic matrix with a
  random matrix.
\newblock {\em arXiv:1603.04066}.

\bibitem{SmallTX}
F.~Yang.
\newblock The smallest singular value of deformed random rectangular matrices.
\newblock {\em arXiv:1702.04050}.

\bibitem{BaiYin_large}
Y.~Q. Yin, Z.~D. Bai, and P.~R. Krishnaiah.
\newblock On the limit of the largest eigenvalue of the large dimensional
  sample covariance matrix.
\newblock {\em Probability Theory and Related Fields}, 78(4):509--521, 1988.

\end{thebibliography}

\end{document}